\documentclass[11pt
]{article}

\usepackage[all]{xy}

\usepackage[latin1]{inputenc}
\usepackage{amsfonts}
\usepackage{amsmath}
\usepackage{amssymb}
\usepackage{amscd}
\usepackage{amsthm}
\usepackage{indentfirst}

\usepackage{epsfig}

\newtheorem{thm}{Theorem}[section]
\newtheorem{lemma}[thm]{Lemma}
\newtheorem{prop}[thm]{Proposition}
\newtheorem{coroll}[thm]{Corollary}

\theoremstyle{definition}

\newcommand{\R}{{\mathbb{R}}}

\newcommand{\Z}{{\mathbb{Z}}}
\newcommand{\N}{{\mathbb{N}}}
\newcommand{\C}{{\mathbb{C}}}

\newcommand{\cA}{{\mathcal{A}}}
\newcommand{\cB}{{\mathcal{B}}}

\newcommand{\cM}{{\mathcal{M}}}

\newcommand{\fc}{{:\ }}

\newcommand{\ol}{\overline}

\newcommand{\wh}{\widehat}

\newcommand{\tb}{\textbf}

\DeclareMathOperator{\im}{im}
\DeclareMathOperator{\id}{id}

\DeclareMathOperator{\Conf}{Conf}

\DeclareMathOperator{\Symp}{Symp}

\DeclareMathOperator{\Cont}{Cont}
\DeclareMathOperator{\Diff}{Diff}

\title
{On the contact mapping class group of the contactization of the $A_m$-Milnor fiber}
\author{Sergei Lanzat and Frol Zapolsky}
\date{}

\begin{document}

\renewcommand{\labelenumi}{(\roman{enumi})}

\maketitle

\begin{abstract}We construct an embedding of the full braid group on $m+1$ strands $B_{m+1}$, $m \geq 1$, into the contact mapping class group of the contactization $Q \times S^1$ of the $A_m$-Milnor fiber $Q$. The construction uses the embedding of $B_{m+1}$ into the symplectic mapping class group of $Q$ due to Khovanov and Seidel, and a natural lifting homomorphism. In order to show that the composed homomorphism is still injective, we use a partially linearized variant of the Chekanov--Eliashberg dga for Legendrians which lie above one another in $Q \times \R$, reducing the proof to Floer homology. As corollaries we obtain a contribution to the contact isotopy problem for $Q \times S^1$, as well as the fact that in dimension $4$, the lifting homomorphism embeds the symplectic mapping class group of $Q$ into the contact mapping class group of $Q \times S^1$.
\end{abstract}

\section{Introduction and main result}\label{sec:intro}

Let $(V,\xi)$ be a contact manifold with a cooriented contact structure, let $\Cont(V,\xi)$ be its group of coorientation-preserving compactly supported contactomorphisms, and let $\Cont_0(V,\xi) \subset \Cont(V,\xi)$ be the identity component. In this note we study the contact mapping class group
$$\pi_0\Cont(V,\xi) = \Cont(V,\xi)/\Cont_0(V,\xi)$$
when $V$ is the contactization of the $A_m$-Milnor fiber. For $m,n \in \N$ with $n \geq 2$ the $2n$-dimensional Milnor fiber is defined as the complex hypersurface
$$Q \equiv Q_m = \{w=(w_0,\dots,w_n) \,| \, w_0^2 + \dots + w_n^2 + w_n^{m+1} = 1\} \subset \C^{n+1}\,.$$
The symplectic form $\omega \equiv \omega_Q$ is the restriction of the standard symplectic form $\omega_0  = \tfrac i 2 dw \wedge d\ol w$ on $\C^{n+1}$. We let $\theta \equiv \theta_Q$ be the restriction of the standard primitive $\theta_0 = \tfrac i 4 (w\,d\ol w - \ol w \, dw)$. The contactization has $V = Q \times S^1$ as the underlying manifold and $\alpha = dz - \theta$ as the contact form, where $z$ is the coordinate on $S^1$, giving the contact distribution $\xi = \ker \alpha$. Let $B_{m+1}$ denote the full braid group on $m+1$ strands. Our main result is
\begin{thm}\label{thm:main_thm} There is a group embedding
$$\rho_{\Cont} \fc B_{m+1} \to \pi_0\Cont(V,\xi)\,.$$
\end{thm}

This embedding is constructed as follows. Since $Q$ has one end, that is it is exhausted by compacts with connected complement, and $H^1(Q;\R) = 0$, for a compactly supported symplectomorphism $\phi$ of $(Q,\omega)$ there is a unique function $H_\phi \in C^\infty(Q)$ with compact support such that $\phi^*\theta = \theta + dH_\phi$. We let the contact lift of $\phi$ to $V$ be defined by
$$\widehat \phi (p,z) = (\phi(p),z + H_\phi(p) \mod 1)\,.$$
If $\psi$ is another compactly supported symplectomorphism, then it can be checked that $H_{\phi \circ \psi} = H_\psi + H_{\phi}\circ \psi$ and therefore that
$$\widehat{}\; \fc \Symp(Q,\omega) \to \Cont(V,\xi)$$
is a group homomorphism, where $\Symp(Q,\omega)$ is the group of compactly supported symplectomorphisms of $Q$. This homomorphism is also continuous and therefore induces the lifting homomorphism on mapping class groups
$$\widehat{}\; \fc \pi_0\Symp(Q,\omega) \to \pi_0\Cont(V,\xi)\,.$$
In \cite{Khovanov_Seidel_Quivers_HF_braid_gp_actions} Khovanov--Seidel construct an embedding
$$\rho_{\Symp} \fc B_{m+1} \to \pi_0\Symp(Q,\omega)\,.$$
Our embedding is obtained by composing it with the lifting homomorphism:
$$\rho_{\Cont}:=\widehat {}\;\circ \rho_{\Symp} \fc B_{m+1} \to \pi_0\Cont(V,\xi)\,.$$
Our main contribution is showing that this composed homomorphism remains injective.

When $n=2$, \cite{Wu_Exact_Lags_A_n_surf_singularities} Wu showed that the embedding $\rho_{\Symp}$ is in fact onto and therefore an isomorphism. We thus obtain
\begin{coroll}
For $n=2$, the lifting homomorphism
$$\widehat{}\; \fc \pi_0\Symp(Q,\omega) \to \pi_0\Cont(V,\xi)$$
is an embedding. \qed
\end{coroll}

The contact isotopy problem, analogous to the symplectic isotopy problem (see for instance \cite{Seidel_PhD_thesis}), is about the existence of contactomorphisms which are smoothly isotopic to the identity but not so through contactomorphisms. Let us first make some remarks about the symplectic case. Recall that given a Lagrangian embedding $\ell \fc S^n \to P$, where $P$ is a symplectic manifold, one can define a symplectomorphism of $P$ with support in an arbitrarily small neighborhood of $K = \im \ell$, called the generalized Dehn twist \cite{Seidel_PhD_thesis}, which we denote $\tau_\ell$, or if the embedding is clear from the context, $\tau_K$. It is known that the following symplectomorphisms are smoothly isotopic to the identity by an isotopy supported in a small neighborhood of $K$: when $n=2$, $\tau_K^2$ \cite{Seidel_PhD_thesis, Seidel_Lag_2_spheres_symp_knotted, Seidel_Graded_Lags}; when $n = 6$, $\tau_K^2$ using the cross product on $\R^7$ \cite{Frauenfelder_Schlenk_Vol_growth_component_Dehn_twist}; when $n$ is even, $\tau_K^8$ \cite{Frauenfelder_Schlenk_Vol_growth_component_Dehn_twist}. See also the references therein.

\begin{coroll}\label{cor:ct_isotopy_problem_V}Let $G \triangleleft B_{m+1}$ be the pure braid group $P_{m+1}$ if $n = 2,6$, otherwise let $G$ be the normal subgroup of $B_{m+1}$ generated by the eighth powers of the Artin generators $\sigma_1,\dots,\sigma_m$. Then the embedding $\rho_{\Cont}$ maps $G$ to contact classes which are smoothly trivial. In other words $\rho_{\Cont}(G)$ is contained in the kernel of the natural forgetful homomorphism
$$\beta_{\Cont}\fc \pi_0\Cont(V,\xi) \to \pi_0\Diff^+(V)\,.$$
\end{coroll}
\noindent Here $\Diff^+(V)$ is the group of orientation-preserving diffeomorphisms of $V$ with compact support.
\begin{proof}[Proof of Corollary \ref{cor:ct_isotopy_problem_V}] Smooth isotopies of $Q$ with compact support lift to smooth isotopies of $V$ with compact support. It remains to note that the normal subgroup of $B_{m+1}$ generated by $\sigma_1^2,\dots,\sigma_m^2$ is none other than the subgroup of pure braids $P_{m+1}$ (as follows from \cite[Lemma 1.8.2]{Birman_Braids_links_MCGs}).
\end{proof}

This corollary is a step toward the solution of the contact isotopy problem for $V$. Regarding compact manifolds see the paper by Massot--Nie\-der\-kr\"u\-ger \cite{Massot_Niederkrueger_Examples_nontrivial_contact_mapping_classes_all_dims}, where they provide examples of compact contact manifolds in every dimension such that for each one of them there are infinitely many pairwise distinct contact classes, which are all smoothly trivial. Our contribution is in a sense complementary to theirs, in that our contact manifolds are noncompact, and we provide a large \emph{subgroup} of contact classes which are smoothly trivial. Interestingly enough, our methods are different as well.

Lastly, let us mention that for the standard contact sphere $S^3$ the homotopy type of the full group of contactomorphisms is known, see for instance \cite{Casals_Spacil_Chern_Weil_thry_Strict_Cont} for a proof. Also, when $V$ is a tensor power of the unit cocircle bundle over a closed orientable surface of positive genus, the corresponding contact mapping class group has been computed by Giroux--Massot \cite{Giroux_Massot_Ct_MCG_Leg_circle_bundles}, also see references therein.

\tb{Organization of the paper.} In Section \ref{sec:prelims} we gather the necessary preliminaries. The main technical tool we use is the dg bimodule associated to a two-component Legendrian link. This is a simplification of the Chekanov--Eliashberg dga, adapted to our geometric situation whereby the components of the link lie one above the other. We introduce all the necessary algebraic definitions, namely the stable tame isomorphism type and linearized homology, analogously to the standard treatment in Legendrian contact homology. Section \ref{sec:prf} contains the proof of Theorem \ref{thm:main_thm}.

\tb{Acknowledgements.} We wish to thank Mohammed Abouzaid, Fr\'e\-d\'e\-ric Bourgeois, Baptiste Chantraine, Yasha Eliashberg, Ailsa Keating, Patrick Massot, Dusa McDuff, Cedric Membrez, and Leonid Polterovich for discussions of the results of this paper and for their interest. We are partially supported by grant number 1281 from the GIF, the German--Israeli Foundation for Scientific Research and Development. FZ is also partially supported by grant number 1825/14 from the Israel Science Foundation. Specials thanks are due to Peter Albers, out of a discussion with whom an idea was born which led to this paper; it is part of an ongoing joint research project, which is funded by the aforementioned grant from the GIF.

\section{Preliminaries}\label{sec:prelims}

Throughout this section $(P,\omega = d\theta)$ is an exact symplectic $2n$-manifold with one end and $H^1(P;\R) = 0$. We assume in addition that there is an $\omega$-compatible almost complex structure $I_0$ on $P$ such that $(P,\omega,I_0)$ has finite geometry at infinity \cite{Ekholm_Etnyre_Sullivan_LCH_P_times_R}.

\subsection{Contactization of exact symplectic manifolds}

The contactization of $(P,d\theta)$ is $(P \times \R,\xi = \ker \alpha)$ where the contact form $\alpha = dz - \theta$, $z$ being the $\R$-coordinate. Note that the Reeb field of $\alpha$ is $R_\alpha = \partial_z$. We have the canonical Lagrangian projection
$$\Pi \fc P \times \R \to P\,.$$
Recall that an $n$-dimensional submanifold $L \subset P \times \R$ is called Legendrian if it is tangent to the contact distribution $\xi$, or equivalently $\alpha|_L \equiv 0$. A Reeb chord of $L$ is by definition a curve $c \fc [0,T] \to P \times \R$, where $T > 0$, such that $\dot c = R_\alpha \circ c$ and the endpoints $c(0),c(T)$ lie on $L$.

The restriction $\Pi|_L$ is an immersion. If $L$ is a closed submanifold and $\Pi|_L$ is injective, then $K=\Pi(L)$ is a Lagrangian submanifold of $P$ and $\Pi|_L \fc L \to K$ is a diffeomorphism. We refer to $L$ as a Legendrian lift of $K$. Note that $K$ is an exact Lagrangian, that is there is $F \in C^\infty(K)$ with $dF = \theta|_K$. Indeed, we can take $F = z \circ (\Pi|_L)^{-1}$. Conversely, if $K$ is an exact Lagrangian submanifold and $\theta|_K = dF$ for some $F \in C^\infty(K)$, then a Legendrian lift $L$ can be defined as
$$L = \{(p,F(p))\,|\, p \in K\} \subset P \times \R\,.$$
Note that Legendrian lifts are not unique, and that any two Legendrian lifts of the same connected exact Lagrangian differ by a shift in the $z$ direction.

The lifting homomorphism \;$\wh{}\;$ described in Section \ref{sec:intro} can be defined in the same way for $P$, and it results in a map
$$\wh{}\; \fc \Symp(P,\omega) \to \Cont(P\times \R,\xi)\,,$$
where $\Cont(P \times \R,\xi)$ is the set of contactomorphisms having support in a set of the form $C\times \R$, where $C \subset P$ is compact (it may depend on the contactomorphism). It can similarly be seen that it is a homomorphism. Moreover if $\phi \in \Symp(P,\omega)$ and $L \subset P \times \R$ is a Legendrian lift of $K \subset P$, then $\wh\phi(L)$ is a Legendrian lift of $\phi(K)$.

There is another version of contactization. Namely, consider the manifold $P \times S^1$. Letting $z$ be the coordinate on $S^1$, we likewise can define a contact form $\alpha = dz - \theta$ on $P \times S^1$, giving rise to the contact structure $\xi = \ker \alpha$. We similarly have the lifting homomorphism
$$\wh{}\; \fc \Symp(P,\omega) \to \Cont(P \times S^1,\xi)\,,$$
where now $\Cont(P\times S^1,\xi)$ stands for the group of contactomorphisms of $P \times S^1$ with compact support.

\subsection{Lagrangian Floer homology}

We briefly recall the basic definitions and notation for Lagrangian Floer theory for exact Lagrangians. The reader is referred to the original papers \cite{Floer_Morse_thry_Lagr_intersections, Floer_unregularized_grad_flow_symp_action} for more detailed definitions and proofs. Throughout we work with $\Z_2$ as the coefficient ring. 

Let $L_0,L_1 \subset P$ be two closed connected exact Lagrangian submanifolds, which we for simplicity assume to be simply connected. Assume they intersect transversely. Put
$$CF_*(L_0,L_1) = \bigoplus_{p \in L_0 \cap L_1}\Z_2\cdot p\,.$$
The grading is given by the Maslov--Viterbo index \cite{Viterbo_Intersections_sous_varietes_lags_fonctionnelles_action_indice_systemes_hams}. It is a relative $\Z$-grading, meaning that for $p,q \in L_0,L_1$ the difference of their degrees $|p| - |q| \in \Z$ is well-defined.

Fix a time-dependent compatible almost complex structure $J_t$ on $P$, which coincides with $I_0$ outside a compact subset, and for $p,q \in L_0 \cap L_1$ consider the moduli space $\cM(p,q)$ of $J$-holomorphic strips with boundaries on $L_0,L_1$, asymptotic to $p,q$ (modulo the $\R$-action). For generic $J$ this is a smooth manifold of dimension $|p| - |q| - 1$. Moreover the moduli space is compact in dimension zero. We can then define the Floer boundary operator via
$$\partial (p) = \sum_{q:|p| - |q| = 1}\#_2 \cM(p,q)q\,.$$
It is known that $\partial^2 = 0$. The resulting homology is the Floer homology $HF_*(L_0,L_1;J)$.

The following properties are well-known:
\begin{itemize}
 \item The Floer homology $HF_*(L_0,L_1;J)$ is independent of $J$, and therefore we will write $HF_*(L_0,L_1)$.
 \item If $\{\phi_t\}_{t\in[0,1]}$ is a Hamiltonian isotopy, then there are induced canonical isomorphisms
 $$HF_*(\phi_1(L_0),L_1) \simeq HF_*(L_0,L_1) \simeq HF_*(L_0,\phi_1(L_1))\,.$$
\end{itemize}
The latter property allows us to define the Floer homology of two Lagrangians which do not necessarily intersect transversely. Indeed, it suffices to perturb one of them using a small Hamiltonian perturbation. The resulting homology is then independent of the perturbation used.

\subsection{The $A_m$-Milnor fiber}\label{subsec:Milnor_fiber}

Let $m,n \in \N$ with $n \geq 2$. Recall the definition of the $2n$-dimensional $A_m$-Milnor fiber $Q = Q_m$ in Section \ref{sec:intro}. The following properties are well-known:
\begin{itemize}
 \item $Q$ is simply connected. In fact, it is homotopy equivalent to the wedge of $m$ spheres of dimension $n$ \cite{Milnor_Singular_pts_cx_hypersurfaces, AGLV_Singularity_theory_I}.
 \item The first Chern class of $Q$ vanishes \cite{Seidel_Lag_2_spheres_symp_knotted, Seidel_Graded_Lags}.
 \item Note that $\omega = -d(df \circ i)$, where $f \fc Q \to \R$ is the restriction of the function $\tfrac 1 4|w|^2$ on $\C^{n+1}$ and $i$ is the complex structure induced on $Q$ from $\C^{n+1}$. In particular $(Q,\omega,i)$ is a Stein manifold \cite{Eliashberg_Gromov_Convex_symp_mfds}. It follows that it has finite geometry at infinity \cite{Eliashberg_Thurston_Confoliations, Ekholm_Etnyre_Sullivan_LCH_P_times_R}.
\end{itemize}

In \cite{Seidel_Lag_2_spheres_symp_knotted, Seidel_Graded_Lags} Seidel showed that $Q = Q_m$ contains an $A_m$-configura\-tion of Lagrangian spheres $K_1,\dots,K_m$. More precisely, there are Lagrangian embeddings $\ell_j \fc S^n \hookrightarrow Q$, such that their images $K_j = \ell_j(S^n)$ satisfy
$$|K_i\cap K_j| = \left\{ \begin{array}{ll} 0 & \text{if }|i-j| > 1 \\ 1 & \text{if }|i - j| = 1\end{array}\right.$$
for $i \neq j$.

For our purposes we use the construction of such an $A_m$-configuration by Khovanov--Seidel \cite[Section 6c]{Khovanov_Seidel_Quivers_HF_braid_gp_actions} via so-called matching cycles (\cite{Seidel_The_Book}). Namely, consider the projection $\zeta \fc Q \to \C$ given by $\zeta(w_0,\dots,w_n) = w_n$. Next, let $\Delta = \{\zeta_i\}_{i=1}^{m+1} \subset \C$ be the set of roots of unity of order $m+1$ and $\{b_i = [\zeta_i,\zeta_{i+1}]\}_{i=1}^m$ be the set of straight segments connecting $\zeta_i$ with $\zeta_{i+1}$. We have
$$\zeta^{-1}(c) \simeq \left\{\begin{array}{ll}T^*S^{n-1} & \text{if }c\neq \zeta_i \text{ for any }i\\ T^*S^{n-1}/0_{S^{n-1}} & \text{if }c = \zeta_i\text{ for some }i\end{array}\right.$$
Here $T^*S^{n-1}/0_{S^{n-1}}$ means that the zero section has been collapsed to a point. Denote by $\Sigma_c \subset \zeta^{-1}(c)$ the zero section or the point corresponding to the collapsed zero section. For every $i = 1,\dots,m$ put $K_i = \bigsqcup_{c \in b_i}\Sigma_c$.

\subsection{Generalized Dehn twists and braids}

Given a Lagrangian embedding $\ell \fc S^n \to P$, we have the associated generalized Dehn twist along $K = \ell(S^n)$, which we denoted $\tau_K$. This symplectomorphism is trivial outside a tubular neighborhood of $K$ and its restriction to $K$ is the antipodal involution of $S^n \simeq K$, see \cite[Section 6]{Seidel_Lag_2_spheres_symp_knotted}, \cite[Section 5]{Seidel_Graded_Lags}.

Fix $m \in \N$. Recall the configuration $\Delta,b_1,\dots,b_m \subset \C$ above and let $K_1,\dots,K_m$ be the corresponding $A_m$-configuration of Lagrangian spheres in $Q$. Let $\Diff(\C;\Delta)$ be the group of compactly supported diffeomorphisms of $\C$ fixing $\Delta$ as a set. One can identify $B_{m+1}$ with $\pi_0\Diff(\C;\Delta)$ by sending $\sigma_i$ to the isotopy class $[t_{b_i}]$, where $t_{b_i} \in \Diff(\C;\Delta)$ is the right-handed half-twist along $b_i$, for $i = 1,\dots,m$.

In \cite[Section 6b]{Khovanov_Seidel_Quivers_HF_braid_gp_actions} Khovanov--Seidel constructed a monomorphism of groups
$$\rho_{\Symp} \fc B_{m+1} \simeq \pi_0\Diff(\C;\Delta) \rightarrow \pi_0\Symp(Q,\omega)\,,$$
by putting $\rho_{\Symp}(\sigma_i) = [\tau_{K_i}]$. This in fact is the monodromy map of a canonical symplectic fibration over the configuration space $\Conf_{m+1}(\C)$ with fiber $(Q,\omega)$. The injectivity of $\rho_{\Symp}$ follows from the beautiful
\begin{thm}[\cite{Khovanov_Seidel_Quivers_HF_braid_gp_actions}, Theorem 1.3]\label{thm:Khovanov_Seidel} Let $\sigma \in B_{m+1}$ correspond to the class of $f_\sigma \in \Diff(\C;\Delta)$. Then for any representative $\phi_\sigma \in \Symp(Q,\omega)$ of the class $\rho_{\Symp}(\sigma) \in \pi_0\Symp(Q,\omega)$ we have
$$\dim_{\Z_2}HF_*(K_i,\phi_\sigma(K_j)) = 2I(b_i,f_\sigma(b_j))\,,$$
where $I(b_i,f_\sigma(b_j))$ is the geometric intersection number.\footnote{Defined \emph{ibid.}}
\end{thm}

%
%
%

\subsection{Legendrian contact homology}

\subsubsection{The Chekanov--Eliashberg dga of a Legendrian knot}

The Chekanov--Eliashberg differential graded algebra (dga) of a Legendrian knot $L \subset P \times \R$ was constructed by Ekholm--Etnyre--Sullivan \cite{Ekholm_Etnyre_Sullivan_LCH_R2n_plus_1, Ekholm_Etnyre_Sullivan_LCH_P_times_R}. Here we recall the basic notions and construction and refer the reader to the aforementioned papers and references therein for details. We assume throughout that our Legendrians are closed and simply connected.

The technical assumptions one needs to impose are the following:
\begin{itemize}
 \item $L$ is chord-generic, meaning $\Pi|_L$ is an immersion with transverse double points;
 \item There is an $\omega$-compatible almost complex structure $J$ on $P$ which is adapted to $L$, meaning that $(P,\omega,J)$ has finite geometry at infinity,\footnote{This $J$ is required to coincide with $I_0$ outside of a compact subset.} and there is a coordinate system around each double point of $\Pi(L)$ in which $J$ coincides with the standard almost complex structure on $\C^n$;
 \item $L$ is admissible with respect to $J$ meaning that around each double point of $\Pi(L)$ the two branches meeting at it are real analytic submanifolds in the above coordinate system.
\end{itemize}
Note that when $L$ is chord-generic, its Reeb chords are graded by the Maslov index taking values in $\Z$. For a Reeb chord $a$ we let $|a|$ be its degree. We denote words on Reeb chords by an arrow, for instance $\vec b = b_1\dots b_k$. Such words are assigned a degree via $|b_1\dots b_k| = \sum_j |b_j|$.

For Reeb chords $a,b_1,\dots,b_k$ of $L$ (where we allow $k=0$), we have the moduli space $\cM(a;\vec b)$, where $\vec b = b_1\dots b_k$. This is the set, modulo conformal reparametrization, of $J$-holomorphic disks $u \fc (D_k,\partial D_k) \to (P,\Pi(L))$, where $D_k \subset \C$ is the closed unit disk with $k+1$ boundary punctures $z_0,\dots,z_k$, and where the map $u$ is required to extend to a continuous map $\wh u \fc (D^2,\partial D^2) \to (P,\Pi(L))$, satisfying $\wh u(z_0) = \Pi(a)$, $\wh u(z_j) = \Pi(b_j)$ for all $j$, and to lift to a continuous map $(D_k,\partial D_k) \to (P \times \R, L)$. Moreover $a$ is a positive puncture of $u$ while the $b_j$ are all negative punctures. Under the above technical assumptions this is a smooth manifold of dimension $|a| - |\vec b| - 1$, provided this number is at most $1$, for an open and dense set of $J$ which are adapted to $L$ and for which $L$ is admissible; the same holds true if we fix $J$ adapted to $L$ and generically perturb $L$ among Legendrians $L'$ for which $J$ remains adapted and such that $L'$ is admissible for $J$. We refer to $J$ such that this holds regular. Whenever $\cM(a;\vec b)$ is zero-dimensional, it is compact, and therefore a finite set of points.

Fix a chord-generic $L$ and a regular compatible almost complex structure $J$ which is adapted to $L$, and for which $L$ is admissible. The Chekanov--Eliashberg dga associated to $L$ is defined as follows. The underlying algebra $\cA_L$ is the semi-free unital associative algebra over $\Z_2$ with generators the Reeb chords of $L$. It is graded by $|\cdot|$. The differential is defined as follows: for a Reeb chord $a$ of $L$ put
$$\partial_{L,J}(a) = \sum_{\vec b:|a| - |\vec b| = 1}\#_2\cM(a;\vec b)\vec b\,.$$
It is proved in \cite{Ekholm_Etnyre_Sullivan_LCH_R2n_plus_1, Ekholm_Etnyre_Sullivan_LCH_P_times_R} that $\partial_{L,J}^2 = 0$. It follows from the definition that $\partial_{L,J}$ has degree $-1$.

There are notions of a tame isomorphism and stabilization of dgas, see \cite{Chekanov_Differential_algebra_Leg_links}. Two dgas are called stable tame isomorphic if, after being stabilized a number of times, they become tame isomorphic. If $(L_i,J_i)$, $i=0,1$ are pairs of chord-generic Legendrians and adapted almost complex structures with respect to which the Legendrians are admissible, such that $L_0,L_1$ are Legendrian isotopic, then the dgas $(\cA_{L_0},\partial_{L_0,J_0})$, $(\cA_{L_1},\partial_{L_1,J_1})$ are stable tame isomorphic \cite{Ekholm_Etnyre_Sullivan_LCH_R2n_plus_1, Ekholm_Etnyre_Sullivan_LCH_P_times_R}.

If $L$ is any Legendrian, there exists an arbitrarily small Legendrian isotopy connecting it to a chord-generic $L'$ such that there is an adapted regular $J$ with respect to which $L'$ is admissible \cite[Lemma 2.7]{Ekholm_Etnyre_Sullivan_LCH_P_times_R}. It follows that we can associate a stable tame isomorphism class of dgas to any Legendrian submanifold of $P$.

An augmentation of a semi-free dga $(\cA,\partial)$ with generating set $\mathsf A$ is a dga morphism $\epsilon \fc (\cA,\partial) \to (\Z_2,0)$. Given an augmentation $\epsilon$, the corresponding linearized complex is defined as follows. Its underlying vector space is $\cA^{(1)} = \Z_2 \otimes \mathsf A$, while the linearized differential is given by $\partial^\epsilon = \pi_1 \circ \phi^\epsilon \circ \partial$, where $\pi_1 \fc \cA \to \cA^{(1)}$ is the projection onto the subspace generated by words of length $1$, while $\phi^\epsilon$ is the graded linear automorphism determined by $\phi^\epsilon(1) = 1$ and $\phi^\epsilon(a) = a + \epsilon(a)$ for each $a \in \mathsf A$. It can be shown that $\partial^\epsilon$ squares to zero. The homology of the resulting complex $(\cA^{(1)},\partial^\epsilon)$ is the linearized homology associated to $\epsilon$. Chekanov proved \cite{Chekanov_Differential_algebra_Leg_links} that the set of linearized homologies associated to all possible augmentations of a dga is an invariant of its stable tame isomorphism type.

One special case is important for us. Assume $L$ is a Legendrian lift of a Lagrangian submanifold $K \subset P$. The set of Reeb chords of $L$ is of course empty, therefore the corresponding dga is trivial: $(\cA_L,\partial_L) \equiv (\Z_2,0)$. It follows that the linearized homology (associated to the unique augmentation) is zero.

Assume now that $L'$ is chord-generic, Legendrian isotopic to $L$, and that $J$ is an adapted regular almost complex structure with respect to which $L'$ is admissible. Then the associated dga $(\cA_{L'},\partial_{L',J})$ is stable tame isomorphic to $(\Z_2,0)$ and therefore the set of its linearized homologies includes only one element, namely the zero vector space.

\subsubsection{The dg bimodule of a two-component Legendrian link}

If $L,L' \subset P \times \R$ are Legendrians, we say that $L'$ is above $L$, and write $L' > L$, if $z(L') > z(L)$. We call the pair $(L,L')$ chord-generic if the Legendrian link $L \cup L'$ is. In this subsection we describe an invariant of the pair $(L,L')$, in the spirit of the Chekanov--Eliashberg dga. This will take the form of a semi-free differential graded bimodule over the dg algebras $\cA_L$, $\cA_{L'}$.

The idea to use the dga of the union of two Legendrians which lie above one another appears, for instance, in \cite{Ekholm_Etnyre_Sabloff_Duality_exact_seq_LCH}. Dg bimodules we use here appeared, for instance, in \cite{DimitroglouRizell_Golovko_Estimating_num_Reeb_chords_reps_char_algebra}.

Fix a chord-generic pair $L,L'$ such that $L' > L$. Let $\mathsf A,\mathsf A', \mathsf B$ be the sets of Reeb self-chords of $L,L'$, and of mixed Reeb chords from $L$ to $L'$, respectively. For any two chords $b,b' \in \mathsf B$ the difference of their degrees $|b| - |b'| \in \Z$ is well-defined using the Maslov--Viterbo index. Since an absolute grading on $\mathsf B$ depends on the choice of a normalization, and since it plays no role in our constructions, we say that $\mathsf B$ is graded by a $\Z$-torsor, or simply that it is graded. Graded maps are those ones respecting this grading.

We let $\cB_{L,L'}$ be the $\Z_2$-vector space generated by monomials of the form $\vec a b \vec a'$, where $\vec a,\vec a'$ are words (possibly empty) on the alphabets $\mathsf A,\mathsf A'$, while $b \in \mathsf B$. This is the semi-free graded $(\cA_L,\cA_{L'})$-bimodule with generating set $\mathsf B$, where the grading is defined via $|\vec a b \vec a'| = |\vec a| + |b| + |\vec a'|$.

The notion of an $\omega$-compatible almost complex structure $J$ adapted to $L \cup L'$ is the same as above; similarly we can define what it means for $L \cup L'$ to be admissible with respect to $J$. We conclude that the moduli space of disks $\cM(c;\vec d)$ can be defined as above. The following elementary observation is crucial for our construction.
\begin{lemma}\label{lem:Legs_one_above_another}
Let $c,d_1,\dots,d_k$ be Reeb chords of $L \cup L'$. Then $\cM(c;\vec d)$ is empty unless
\begin{itemize}
 \item $c \in \mathsf A$ and $d_1,\dots,d_k \in \mathsf A$ or
 \item $c \in \mathsf A'$ and $d_1,\dots,d_k \in \mathsf A'$ or
 \item $c \in \mathsf B$ and there is a unique $1 \leq j \leq k$ such that $d_1,\dots,d_{j-1} \in \mathsf A$, $d_j \in \mathsf B$, and $d_{j+1},\dots,d_k \in \mathsf A'$.
\end{itemize}
\end{lemma}
\begin{proof}
This follows from the assumption that $c$ is the unique positive puncture and that $L'$ is above $L$.
\end{proof}

It follows that we can define the differential $\partial_{L,L',J}$ on $\cB_{L,L'}$ by the same formula
$$\partial_{L,L',J}(b) = \sum_{\vec a,c,\vec a':|b| - |\vec a c \vec a'| = 1} \#_2 \cM(b;\vec a c \vec a')\vec a c \vec a'\,.$$
Using the basic Lemma \ref{lem:Legs_one_above_another}, we can see that the claim that $\partial_{L,L',J}^2 = 0$ can be proved in the usual way. Also, from the definition it follows that this differential has degree $-1$. This completes the definition of the semi-free dg $(\cA_L,\cA_{L'})$-bimodule $(\cB_{L,L'},\partial_{L,L',J})$.

We can define the notion of a tame isomorphism between two semi-free graded bimodules. First, note that a graded algebra isomorphism $\cA \to \cA_1$ induces a graded isomorphism between the semi-free $(\cA,\cA')$- and $(\cA_1,\cA')$-bimodules generated by $\mathsf B$. Similarly we have a bimodule isomorphism induced by a graded algebra isomorphism $\cA' \to \cA'_1$. We say that a graded bimodule automorphism $\phi \fc \cB \to \cB$ of a semi-free $(\cA,\cA')$-bimodule $\cB$ generated by $\mathsf B$ is $\mathsf B$-elementary if there is $b \in \mathsf B$ such that $\phi(\vec a c \vec a') = \vec a c \vec a'$ if $c \neq b$ and $\phi(\vec a b \vec a') = \vec a b \vec a' + b'$ where $b'$ belongs to the semi-free $(\cA,\cA')$-bimodule generated by $\mathsf B - \{b\}$. A graded bimodule automorphism of $\cB$ is $\mathsf B$-tame if it is a composition of $\mathsf B$-elementary automorphisms. Given a semi-free graded $(\cA_0,\cA'_0)$-bimodule $\cB_0$ and a semi-free graded $(\cA_1,\cA_1')$-bimodule $\cB_1$, we call a map $\cB_0 \to \cB_1$ a tame isomorphism if it is given by a composition of isomorphisms induced by tame isomorphisms $\cA_0 \simeq \cA_1$, $\cA_0'\simeq \cA_1'$, and $\mathsf B$-tame automorphisms followed by a renaming of generators. Such an isomorphism is a tame isomorphism of dg bimodules if it is an isomorphism of dg bimodules (over the corresponding dga isomorphisms of their algebras) and it is tame as an isomorphism of graded bimodules.

Similarly we can introduce the notion of stabilization of a semi-free dg $(\cA,\cA')$-bimodule $\cB$, as follows. One kind of stabilization comes from stabilizing the dgas $\cA,\cA'$, in an obvious manner. Another kind of stabilization comes from adding two generators of degree difference $1$ to the generating set of $\cB$, where the stabilized differential maps one generator to the other.

Finally, two dg bimodules are stable tame isomorphic, if they become tame isomorphic (over the corresponding tame isomorphisms of their generating dgas) after being stabilized a number of times.

If $(L_i,L_i',J_i)$, $i=0,1$ are triples consisting of Legendrians and almost complex structures, such that for $i=0,1$ we have that $L_i'$ lies above $L_i$, $L_i \cup L_i'$ is chord-generic, $J_i$ is adapted to $L_i\cup L_i'$, and $L_i \cup L_i'$ is admissible with respect to $J_i$, and such that there is an isotopy of Legendrian links $L_t \cup L_t'$ connecting $L_0 \cup L_0'$ to $L_1 \cup L_1'$, such that $L_t'$ lies above $L_t$ for all $t$, then, using arguments similar to the usual ones in Legendrian contact homology, in conjunction with Lemma \ref{lem:Legs_one_above_another}, we can conclude that the semi-free dg $(\cA_{L_0},\cA_{L_0'})$-bimodule $(\cB_{L_0,L_0'},\partial_{L_0,L_0',J_0})$ is stable tame isomorphic to the semi-free dg $(\cA_{L_1},\cA_{L_1'})$-bimodule $(\cB_{L_1,L_1'},\partial_{L_1,L_1',J_1})$.

Similarly to the case of a Legendrian knot, we can associate a stable tame isomorphism type of dg bimodules to any two-component Legendrian link $L \cup L'$ provided that $L' > L$.

Given a semi-free dg $(\cA,\cA')$-bimodule $(\cB,\partial)$ generated by $\mathsf B$ we can similarly define the linearized complex associated to a pair of augmentations $\epsilon,\epsilon'$ of $\cA,\cA'$, respectively. The underlying vector space is $\cB^{(1)} = \Z_2 \otimes \mathsf B$. The linearized differential is defined by $\partial^{\epsilon,\epsilon'} = \pi \circ \phi^{\epsilon,\epsilon'} \circ \partial$, where $\pi \fc \cB \to \cB^{(1)}$ is the obvious projection while $\phi^{\epsilon,\epsilon'}$ is the graded automorphism of $\cB$ induced by the automorphisms $\phi^\epsilon,\phi^{\epsilon'}$ of $\cA,\cA'$, respectively. It follows that $\partial^{\epsilon,\epsilon'}$ squares to zero. The linearized homology associated to $(\epsilon,\epsilon')$ is the homology of $(\cB^{(1)}, \partial^{\epsilon,\epsilon'})$. By an argument analogous to the one used in the case of dgas, it follows that set of linearized homologies is an invariant of the stable tame isomorphism type of the dg bimodule.

We have the following observation.
\begin{prop}Let $L,L'$ be Legendrian lifts of exact Lagrangians $K,K'$ which intersect transversely, such that $L' > L$. Then the set of linearized homologies of the associated dg bimodule $\cB_{L,L'}$ contains a single element canonically isomorphic to $HF_*(K,K')$.
\end{prop}
\begin{proof}
If necessary, perturb $L,L'$ so that there is an almost complex structure $J$ adapted to $L \cup L'$ and so that $L \cup L'$ is admissible with respect to $J$, and so that their projections to $P$ are still embedded. The algebras $\cA_L,\cA_{L'}$ are trivial, therefore admit unique augmentations, and the dg bimodule is just the $\Z_2$-vector space spanned by the set of mixed chords $\mathsf B$, which is in bijection with $K \cap K'$. It can be seen that if $b,b' \in \mathsf B$ and $p=\Pi(b),p'=\Pi(b')$ are the corresponding intersection points of $K,K'$, then there is a canonical identification $\cM(b,b') \simeq \cM(p,p')$, where we use $J$ to define both spaces, and therefore the differential on $\cB_{L,L'} = \Z_2 \otimes \mathsf B \simeq \Z_2 \otimes (K_0 \cap K_1)$ coincides with the Floer differential. It then follows that $\cB_{L,L'}$ coincides with its own linearization, and in turn coincides with the Floer complex of $K,K'$.
\end{proof}
\begin{coroll}\label{cor:linearized_LCH_equal_HF_isotopy}
Let $K,K',K''$ be Lagrangians such that $K$ intersects both $K'$ and $K''$ transversely. Assume $L,L',L''$ are Legendrian lifts of $K,K',K''$, respectively, such that there is a Legendrian isotopy $\{L'_t\}_{t \in [0,1]}$ with $L_0' = L'$, $L_1'=L''$ and $L_t' > L$ for all $t$. Then $HF_*(K,K') \simeq HF_*(K,K'')$. \qed
\end{coroll}

\section{Proof of Theorem \ref{thm:main_thm}}\label{sec:prf}

We must show that if $\sigma \in B_{m+1}$ is such that $\rho_{\Cont}(\sigma) \in \pi_0\Cont(V,\xi)$ is the trivial class, then $\sigma$ is the trivial braid. Fix a representative $\phi_\sigma$ of $\rho_{\Symp}(\sigma)$. Choose a contact isotopy starting at $\id_V$ and ending at $\wh{\phi_\sigma}$, and let $\{\wh\phi_t\}_{t\in[0,1]}$ be its unique lift to a contact isotopy of $Q \times \R$ with support in a set of the form $C\times \R$, with $C \subset Q$ compact.

Recall the Lagrangian spheres $K_1,\dots,K_m \subset Q$ introduced in Subsection \ref{subsec:Milnor_fiber}. Fix $i,j = 1,\dots,m$ and let $L_i,L_j$ be Legendrian lifts of $K_i,K_j$ to $Q \times \R$ such that $\wh\phi_t(L_j) > L_i$ for all $t$. By Corollary \ref{cor:linearized_LCH_equal_HF_isotopy} we see that
$$HF_*(K_i,K_j) \simeq HF_*(K_i,\phi_\sigma(K_j))\,.$$
A similar argument shows that
$$HF_*(K_i,K_j) \simeq HF_*(K_i,\phi_{\sigma^2}(K_j))\,.$$
It follows from Theorem \ref{thm:Khovanov_Seidel} that
$$I(b_i,b_j) = I(b_i,f_\sigma(b_j)) = I(b_i,f_{\sigma^2}(b_j))$$
for all $i,j$, and then Lemma 3.6 in \cite{Khovanov_Seidel_Quivers_HF_braid_gp_actions} implies that $\sigma$ is the trivial braid. This completes the proof.

\bibliography{biblio}
\bibliographystyle{plain}

\end{document}